\documentclass[12pt]{article}

\usepackage{amsfonts,latexsym,amssymb,amsmath,amsthm,authblk,color}

\date{\today}
\def \al{\alpha}

\def \ga{\gamma}
\def \dl{\delta}
\def \ep{\varepsilon}

\def \om{\omega}
\def \Ga{\Gamma}

\def \Om{\Omega}

\def \operatorname#1{\mathop{\rm #1}}

\def\div{\operatorname{div}}
\def\osc{\operatorname{osc}}
\def\osc2{\operatorname{osc^2}}
\def\dist{\operatorname{dist}}

\def\cd{\partial}
\def\esssup{\operatorname{esssup}}

\def\Q0{Q(x_0,t_0,R)}
\def\0{{x_0,t_0,R}}
\def\build#1_#2{\mathrel{\mathop{\kern 0pt#1}\limits_{#2}}}

\newtheorem{theorem}{Theorem}[section]

\newtheorem{proposition}{Proposition}[section]

\newtheorem{definition}{Definition}[section]

\begin{document}

\title{On the $L_1$--stability  for  parabolic   equations \\  with a  supercritical drift term}
\author{Mikhail Glazkov\footnote{Instituto de Matem{\' a}tica Pura e Aplicada, Brazil}, Timofey  Shilkin\footnote{Leipzig University, Germany, and V.A.~Steklov Mathematical Institute, St.-Petersburg, Russia}}
\date{\today}
\maketitle

 \abstract{In this paper we investigate the  existence, uniqueness and stability of weak solutions  of the initial boundary value problem with the Dirichlet boundary conditions for a parabolic  equation with a drift $b\in L_2$. We prove $L_1$-stability of solutions with respect to perturbations of the drift $b$ in $L_2$} in the case if the  drift    satisfies  the ``non-spectral'' condition $\div b\le 0$.

\section{Introduction and Main Results}

\bigskip

Let $\Om\subset \Bbb R^n$ be a bounded Lipschitz domain with a    boundary $\cd \Om$, $n\ge 2$. Assume $T>0$ and denote  $Q_T:=\Om\times (0,T)$.
We consider a problem
\begin{equation}\label{Equation}
\left\{ \quad \gathered  \cd_t u  -\nu\, \Delta u + b\cdot \nabla u   \ = \  f  \qquad \mbox{in} \quad Q_T, \\
u|_{\cd\Om\times (0,T)} \ = \ 0, \qquad \qquad  \qquad \\
u|_{t=0} \ = \ u_0. \qquad \qquad  \qquad  \endgathered \right.
\end{equation}
Here  $u:Q_T \to \Bbb R^n$ is unknown,  $b: Q_T  \to \Bbb R^n $, $f: Q_T\to \Bbb R$ and  $u_0: \Om\to \Bbb R$ are given functions and $\nu>0$ is a given constant.
For minor technical simplifications throughout the whole paper we will always assume the following condition holds:
\begin{equation}\label{Problem_Data}
b\in L_2(Q_T), \qquad u_0\in L_2(\Om), \qquad
  f\in L_1(Q_T).
\end{equation}
Sometimes we impose also the assumption
\begin{equation}\label{RHS_assumption}
f\in   L_2(0,T; W^{-1}_2(\Om)),
\end{equation}
where $W^{-1}_2(\Om)$ is the dual to the Sobolev space $\overset{\circ}{W}{^1_2}(\Om)$ (see notation  for functional spaces at the end of this section). Note that for $n\ge 3$   the space $L_{\frac{2n}{n+2}}(\Om)$ is imbedded into $W^{-1}_2(\Om)$, so \eqref{RHS_assumption} can be viewed   as a bit stronger assumption than $f\in L_1(Q_T)$.

In this paper we are interested in the existence, uniqueness and stability of weak solutions to the problem \eqref{Equation} belonging to the energy class
  \begin{equation}\label{Energy_class}
  u\in L_\infty(0,T; L_2(\Om))\cap  L_2(0,T; \overset{\circ}{W}{^1_2}(\Om)).
  \end{equation}
 For  $u$ satisfying \eqref{Energy_class} one  can consider    the bilinear form
 \begin{equation}\label{Bilinear_form}
 \mathcal B[u,\eta ] \ := \ \int\limits_{Q_T} b\cdot \nabla u\, \eta\,  dxdt, \qquad \eta \in L_\infty(Q_T).
   \end{equation}
 Note that for $u$ belonging to \eqref{Energy_class} the condition $b\in L_2(Q_T)$ is optimal in the sense that the bilinear form is well-defined. Hence    the following definition makes  sense:

\begin{definition} \label{Def_weak_solution}
  Assume $b$, $f$,  $u_0$ satisfy \eqref{Problem_Data}. We say $u$ is a {\it weak solution} to the problem \eqref{Equation} if  $u$ belongs to the class \eqref{Energy_class}
   and  satisfies the following identity
\begin{equation}
\int\limits_{Q_T} \Big( -u\cd_t\eta + \nu \, \nabla u\cdot \nabla \eta \Big)\, dxdt \ + \ \mathcal B[u,\eta] \ = \ \int\limits_\Om u_0(x)\eta(x,0)\, dx \ + \ \int\limits_{Q_T}  f \eta \, dxdt
\label{Integral_Identity}
\end{equation}
 for any $ \eta\in C_0^\infty(\Om\times [0,T))$.
\end{definition}

 By approximation  we can also extend the class of test functions in \eqref{Integral_Identity} to
$$
\eta\in W^{1,1}_2(Q_T)\cap L_\infty(Q_T), \qquad \eta|_{\cd \Om\times (0,T)} =0, \qquad \eta|_{t=T}=0,
$$
where we denote
$$
W^{1,1}_2(Q_T):=\{ \, \eta\in L_2(Q_T):\, \nabla \eta, \, \cd_t \eta \in L_2(Q_T)\, \}.
$$
Note that if
\begin{equation} \label{Div-free_drift}
\div b=0 \quad \mbox{in} \quad \mathcal D'(Q_T) \qquad \mbox{(i.e. in the sense of distributions)}
\end{equation}
 then the convective term in \eqref{Equation} can be represented in the divergence form $b\cdot \nabla u = \div (ub) $ which in principle allows to introduce weak solutions from a wider class than \eqref{Energy_class}.
In this paper we do not assume that $b$ is divergence free. Instead we impose the condition
\begin{equation}
\div b \, \le \, 0 \quad \mbox{in} \quad \mathcal D'(Q_T).
\label{Non-spectral}
\end{equation}
This means that
$$
\int\limits_{Q_T} b\cdot \nabla \eta\, dxdt \,  \ge \, 0, \quad \forall \, \eta\in C_0^\infty(Q_T): \quad \eta\ge 0 \ \mbox{ in } \ Q_T.
$$
Note that for $b\in L_2(Q_T)$ satisfying \eqref{Non-spectral} the quadratic form \eqref{Bilinear_form} is non-negative:
$$
\mathcal B[u,u] \, \ge \, 0, \qquad \forall\, u \in  L_2(0,T; \overset{\circ}{W}{^1_2}(\Om))\cap L_\infty(Q_T),
$$
which means that the drift $b$ satisfying \eqref{Non-spectral} ``shifts'' the elliptic part of the operator in \eqref{Equation} away from the ``spectral area''. That is why sometimes we call  \eqref{Non-spectral} the {\it non-spectral} condition.

We start with the result on the existence of weak solutions in non-spectral case (which is probably known):

\begin{theorem}\label{Existence_Theorem}
Assume   $b$, $u_0$, $f$   satisfy \eqref{Problem_Data},  \eqref{RHS_assumption}, \eqref{Non-spectral}. Then   there exists at least one  weak solution to the problem \eqref{Equation} which satisfies the energy estimate
\begin{equation} \label{Energy_estimate}
\|  u\|_{L_\infty(0,T; L_2(\Om))}  \ + \   \sqrt{\nu}\, \| \nabla u\|_{L_2(Q_T)}  \, \le \, c\, \Big( \|u_0\|_{L_2(\Om)} + \|f \|_{L_2(0,T; W^{-1}_2(\Om))} \Big)
\end{equation}
 with some constant $c>0$ depending only on $n$, $\Om$ and $T$.
 \end{theorem}

Our first main result is the following theorem:

\begin{theorem}\label{Main_Theorem}
 Assume $b$, $u_0$, $f$ satisfy \eqref{Problem_Data}, \eqref{Non-spectral}   and $u$ is a weak solution to \eqref{Equation}.
Then
 $$
  u\in C([0,T]; L_1(\Om))
  $$
  and
for any $0\le t_1\le t_2\le T$
\begin{equation}
\| u (\cdot, t_2)\|_{L_1(\Om)} \ \le \ \| u (\cdot, t_1)\|_{L_1(\Om)} \ + \ \, \int\limits_{t_1}^{t_2} \int\limits_\Om |f(x,t)|\, dxdt.
\label{Decay_L_1_norm}
\end{equation}
\end{theorem}

Note that Theorem \ref{Main_Theorem} says in particular that for drifts $b\in L_2(Q_T)$ satisfying the condition \eqref{Non-spectral} and for the right-hand side  $f\equiv 0$ the $L_1$--norm of weak solutions $u$ to the problem \eqref{Equation} is a non-increasing  function of  time. (Compare it, for example, with the famous Kruzhkov's result \cite{Kruzhkov} for entropy solutions for the scalar conservation laws).

From Theorem \ref{Main_Theorem}  the uniqueness of weak solutions   to the problem \eqref{Equation} in the case of a drift $b\in L_2(Q_T)$  satisfying the non-spectral condition \eqref{Non-spectral} follows immediately:

\begin{theorem}\label{Uniqueness_Theorem}
  Assume    $b \in L_2(Q_T)$  satisfies   \eqref{Non-spectral}.
Then for any  $u_0\in L_2(\Om)$,   $f \in L_1(Q_T)$ the problem \eqref{Equation} has at most one  weak solution in the sense of Definition \ref{Def_weak_solution}.
\end{theorem}

So, for any $b$, $u_0$, $f$ satisfying the conditions \eqref{Problem_Data}, \eqref{RHS_assumption} and \eqref{Non-spectral} there exists a unique weak solution of the problem \eqref{Equation}. Note that  the violation of the condition \eqref{Non-spectral} can destroy the uniqueness for the problem \eqref{Equation}.
In the elliptic case we have obvious examples of non-uniqueness for the drift  $b(x)=\al \frac x{|x|^2}$ belonging to the critical weak Lebesgue space $  L_{n, w}(\Om)$, $\Om:=\{~x\in \Bbb R^n:\, |x|<1\,\}$:
$$
   -\Delta u + n \, \frac x{|x|^2}\cdot \nabla u =0 , \qquad u(x) = c\, (1-|x|^2),
   $$
   $$
   -\Delta u + (n-2)\, \frac x{|x|^2}\cdot \nabla u =0 , \qquad u(x) =c\, \ln |x|.
   $$
 In the parabolic case we can  also construct an example of instant non-uniqueness starting from zero initial data:

\begin{proposition}
  \label{Non-uniqueness}
Assume $\Om=\{\, x\in \Bbb R^n:\, |x|<1\, \}$,  $u_0\equiv 0$, $f\equiv 0$, $\nu=1$.
 Then for any $T>0$ there exist $b\in L_\infty(0,T; L_{n,w}(\Om))$ and $u\not \equiv 0$ such that $u$ is a weak solution to the problem \eqref{Equation}.

\end{proposition}

Note that the uniqueness of weak solutions to the problem \eqref{Equation} is well-known   if the drift  belongs    to the Ladyzhenskaya-Prodi-Serrin class
 \begin{equation}\label{LPS_class}
 b\in L_l(0,T; L_s(\Om)), \qquad \tfrac ns+\tfrac 2l \, \le \,  1, \qquad l\in [2, +\infty], \qquad s\in [n, +\infty].
 \end{equation}
  Indeed, for $b$ satisfying \eqref{LPS_class} the uniqueness can be easily proved with the help of the Gronwall lemma and interpolation inequalities. So, Theorem \ref{Main_Theorem} shows that the non-spectral condition \eqref{Non-spectral} allows to relax the condition \eqref{LPS_class}.

Note that for a non-spectral drift $b\in L_2(Q_T)$ satisfying \eqref{Non-spectral} a weak solution to the problem \eqref{Equation} is bounded providing the data of the problem are sufficiently regular. Namely, we have the following result:

\begin{theorem}\label{Maximum_Princip_Theorem}
 Assume $b\in L_2(Q_T)$ satisfies \eqref{Non-spectral}, $u_0\in L_\infty(\Om)$, $f\in L_s(Q_T)$ with $s>\frac {n+2}2$.  Let $u$ be a weak solution to \eqref{Equation}.
 Then $u\in L_\infty(Q_T)$ and
\begin{equation}\label{Maximum_Princip_Inequality}
\|  u\|_{L_\infty(Q_T)} \, \le \,   \|u_0\|_{L_\infty(\Om)} \ + \ c\, \| f\|_{L_s(Q_T)}
\end{equation}
with some constant $c>0$ depending only on $n$, $\Om$, $T$, $s$ and $\nu$.
\end{theorem}

We emphasize that our Theorem \ref{Maximum_Princip_Theorem} is concerned only with solutions to  the initial boundary value problem  (with sufficiently ``regular'' initial and boundary data)  and its analogue in the local setting for the drift $b\in L_2(Q_T)$  fails even in the case of the divergence-free case \eqref{Div-free_drift}, see  \cite{Albritton}, see also the further discussion of the influence of  the regularity of boundary data on the boundedness of weak solutions inside the domain in \cite{Filonov_Shilkin}.  Note also that for a sufficiently smooth drift  satisfying the condition \eqref{LPS_class}  the estimate \eqref{Maximum_Princip_Inequality} holds even if the condition \eqref{Non-spectral} is violated,   but in this case the constant $c=c_b>0$ in \eqref{Maximum_Princip_Inequality} depends on the norm \eqref{LPS_class} of $b$.

Our second  main result   is the following theorem on the stability of weak solutions:

\begin{theorem}\label{Stability_Theorem} Assume functions $b$, $u_0$, $f$ and $b^m$, $u_0^m$, $f^m$ satisfy \eqref{Problem_Data} and \eqref{Non-spectral} and $f$ satisfies \eqref{RHS_assumption}. Let $u$  and $u^m$ be the unique weak solutions to the problem \eqref{Equation} corresponding to   $b$, $u_0$, $f$ and $b^m$, $u_0^m$, $f^m$  respectively.
 Then
 \begin{equation}\label{Stability_estimate}
\gathered \| u^m - u\|_{L_\infty(0,T; L_1(\Om))}  \ \le   \ \| u^m_0 - u_0\|_{ L_1(\Om)}  +   \, \| f^m - f\|_{L_1(Q_T)}   +  \\  + \,  \frac {c}{\sqrt{\nu}} \,  \| b^m - b\|_{L_2(Q_T)} \, \Big( \|u_0\|_{L_2(\Om)} + \|f \|_{L_2(0,T; W^{-1}_2(\Om))}\Big).
\endgathered
\end{equation}
 with some constant $c>0$ depending only on $n$, $\Om$ and $T$.
In particular, if
$$
b^m\to b \quad \mbox{in} \quad L_2(Q_T), \qquad u_0^m\to  u_0\quad  \mbox{in} \quad L_1(\Om), \qquad  f^m \to f \quad  \mbox{in} \quad L_1(Q_T)
$$
then
$$
u^m \to u \quad \mbox{in} \quad L_\infty(0,T; L_1(\Om)).
$$
\end{theorem}

From our point of view the $L_1$-stability of solutions under perturbations of a drift $b$ in $L_2$  in Theorem \ref{Stability_Theorem} is most interesting, as  for a fixed $b\in L_2(Q_T)$  satisfying \eqref{Non-spectral} the $L_2$-stability of weak solutions to \eqref{Equation} with respect to the right hand side $f$ and the initial data $u_0$ follows directly  from the energy estimate \eqref{Energy_estimate} and $L_\infty$-stability follows from \eqref{Maximum_Princip_Inequality}.
On the other hand, for weak solutions in the energy class  \eqref{Energy_class}, thanks to the obvious identity
$$
b^m\cdot \nabla u^m -  b \cdot \nabla u = b^m \cdot \nabla (u^m-u) + (b^m-b)\cdot \nabla u,
$$
the stability of the solutions to the problem \eqref{Equation} under the perturbation of the drift $b\in L_2(Q_T)$ can be viewed as a problem of stability of solutions to the same  equations with smooth drifts $b^m$ and the right hand sides $(b^m-b)\cdot \nabla u$  belonging to $L_1(Q_T)$. Hence  our estimate \eqref{Decay_L_1_norm} together with the energy estimate \eqref{Energy_estimate} provide \eqref{Stability_estimate}.

Note that  the violation of the condition \eqref{Non-spectral} can destroy the stability  for the problem \eqref{Equation}.  Namely, we have the following counterexample:

\begin{proposition}
  \label{Non-stability} Assume $\Om=\{\, x\in \Bbb R^n:\, |x|<1\, \}$ and
  $$
  b(x) = (n-2)\, \frac{x}{|x|^2}, \qquad b^\ep(x) = b(x)  -  \ep \, \ln|x| \, x ,  \qquad u_0(x) = \ln |x|.
  $$
  Then for any $\ep\ge 0$   the function  $$u^\ep(x,t) \, = \,  e^{\ep t} \,  u_0(x) $$  is a weak solution  to the  problem \eqref{Equation} in $\Om\times (0, +\infty)$ corresponding to the initial data $u_0$, the drift $b^\ep$, $\nu = 1$   and the right hand side $f\equiv 0$. Moreover,  for any $\ep \in (0,1)$ there exists $t_\ep >0$  such that
\begin{equation}
   \|b^\ep -b \|_{L_2( Q_{t_\ep} )} \, < \, \ep , \qquad  \|u^\ep(\cdot, t_\ep) - u_0  \|_{L_1(\Om)} \ \ge  \ \frac{1}{\ep}.
 \label{Norm_growth}
 \end{equation}
 Here $Q_{t_\ep}:= \Om\times (0,t_\ep)$.
\end{proposition}


\bigskip
The study of the weak solvability and properties of weak solutions to drift-diffusion equations has a long history. Very often the problem \eqref{Equation} is considered under the additional restriction of the divergence free drift \eqref{Div-free_drift}
see, for example, \cite{Albritton},  \cite{Lis_Zhang}, \cite{Sem}, \cite{SSSZ}, \cite{Vicol}, \cite{SVZ}, \cite{Zhang} and references there.  In this paper we  demonstrate  that in the case of a supercritical non-divergence free drift $b$    the  non-spectral condition \eqref{Non-spectral}  plays the crucial role for existence, uniqueness  and stability of weak solutions to the problem \eqref{Equation}.

Sometimes the study of the problem \eqref{Equation} is coupled with the study of the dual problem
\begin{equation}\label{Equation_dual}
\left\{ \quad \gathered    -\cd_t w  - \Delta w - \div (b w)  \, = \, g \qquad\mbox{in}\quad  Q_T, \\
w|_{\cd \Om \times (0,T)} \ = \ 0, \qquad \quad \\
w|_{t=T} \ = \ w_0, \qquad  \quad
\endgathered\right.
\end{equation}
where $g:Q_T\to \Bbb R$ and $w_0:\Om\to \Bbb R$ are given functions.
The problem \eqref{Equation_dual} is presented in the divergent form and hence  its  weak solutions, depending on the regularity of $b$,  in principle  can be defined in a wider class than the energy space \eqref{Energy_class}.
For example, in  \cite{Boccardo}  (see also references there) distributional solutions to the problem \eqref{Equation_dual}  are studied.
The $L_1$-stability   for the problem \eqref{Equation_dual} in the one-dimensional case ($n=1$)  was studied in \cite{Sanders}.
Note that thought in the divergence free case  \eqref{Div-free_drift} the problems \eqref{Equation} and \eqref{Equation_dual} are equivalent,  it is not so in  the general case.

The existence of weak solutions in the energy class \eqref{Energy_class} under the assumptions \eqref{Problem_Data},  \eqref{RHS_assumption}, \eqref{Non-spectral} follows directly from the energy estimate in Theorem \ref{Existence_Theorem}. The uniqueness of weak solution under the same assumptions is more delicate. In \cite{Zhang} it was proved for the divergence free case \eqref{Div-free_drift} with the help of duality method  which provides the following estimate of weak solutions to the problem \eqref{Equation}:
\begin{equation}\label{Esimate_via_duality}
\| u \|_{L_1(Q_T)} \, \le \, c\, \Big( \|u_0\|_{L_1(\Om)} + \| f\|_{L_1(Q_T)} \Big)
\end{equation}
with some constant $c>0$ depending only on $n$, $\Om$, $T$ and $\nu$. Here we only briefly sketch the idea assuming all functions to be sufficiently smooth, see   details in \cite{Zhang}, see also the  explanation  of the idea  in the elliptic case in   \cite{Filonov_Shilkin}. Using duality we can represent the $L_1$-norm of $u$ as
$$
\| u\|_{L_1(Q_T)} \ = \ \sup\limits_{g\in L_\infty(Q_T), \, \|g\|_{L_\infty(Q_T)\le 1}} \Big| \int\limits_{Q_T} ug\, dxdt\, \Big|.
$$
Given $g\in L_\infty(Q_T)$ with $\|g\|_{L_\infty(Q_T)}\le 1$, denote by $w_g$ the solution to the problem  \eqref{Equation_dual} with $w_0\equiv 0$.
For the problem \eqref{Equation_dual} similar to \eqref{Maximum_Princip_Inequality} we have
\begin{equation}\label{Maximum_Princip_Inequality_dual}
\| w_g\|_{L_\infty(Q_T)} \, \le \, c\, \|g\|_{L_\infty(Q_T)} \,  \le \, c
\end{equation}
with some constant $c$ depending only on $n$, $\Om$, $T$ and $\nu$. We emphasize   that both estimates \eqref{Maximum_Princip_Inequality} and \eqref{Maximum_Princip_Inequality_dual} with constants independent on $b$ relay crucially on the assumption \eqref{Non-spectral}.
Using integration by parts we obtain
$$
\gathered
\int\limits_{Q_T} ug \, dxdt \ = \ \int\limits_{Q_T} u\Big(  -\cd_t w_g  - \nu\, \Delta w_g - \div (b w_g) \Big)\, dxdt \ = \\ =  \ \int\limits_{Q_T}\Big(  -u\cd_t w_g  +\nu \nabla u\cdot \nabla w_g + b \cdot \nabla u \,  w_g  \Big)\, dxdt \  = \ \int\limits_{\Om } u_0(x)w_g(x,0)\, dx \ + \ \int\limits_{Q_T} fw_g\, dxdt.
\endgathered
$$
Hence we obtain
$$
\gathered
 \Big| \int\limits_{Q_T} ug\, dxdt\, \Big| \ \le \ \Big( \|u_0\|_{L_1(\Om)} + \| f\|_{L_1(Q_T)} \Big) \, \|w_g\|_{L_\infty(Q_T)} \ \le \\ \le \ c\, \Big( \|u_0\|_{L_1(\Om)} + \| f\|_{L_1(Q_T)} \Big)\, \|g\|_{L_\infty(Q_T)} \ \le \  c\, \Big( \|u_0\|_{L_1(\Om)} + \| f\|_{L_1(Q_T)} \Big).
 \endgathered
$$
So, we arrive at \eqref{Esimate_via_duality}
(from the proof it is clear that the norm $\| u\|_{L_1(Q_T)}$ in the left hand side of \eqref{Esimate_via_duality} can be replaced by the norm $\| u\|_{L_s(Q_T)}$ with some sufficiently small $s>1$).

Of cause, the duality method can be used  also  for the proof of the uniqueness of weak solutions in the non-spectral case \eqref{Non-spectral}. Our Theorem \ref{Main_Theorem} gives an alternative proof of the uniqueness. Note also, that comparing with the estimate \eqref{Esimate_via_duality}, our estimate \eqref{Stability_estimate} has the advantage that it is uniform in time. It seems that the duality method does not allow  to obtain such a result.

It can be interesting to compare the uniqueness result for the parabolic equation \eqref{Equation} with the uniqueness theorem for the linear transport equation (of cause, in the case $\nu= 0$ the Dirichlet condition in \eqref{Energy_class} must be replaces by some appropriate boundary condition, for example, non-penetration condition on the boundary). It is well known that the optimal (in a sense) conditions on $b$ which provide the uniqueness of bounded weak solutions for the transport equation are the DiPerna-Lions \cite{DiPerna_Lions}  or Ambrosio  \cite{Ambrosio} conditions. So, in the non-spectral case \eqref{Non-spectral} the inclusion of the viscosity in the equation leads to a significant relaxation in the sense  of the regularity of the drift without destroy of the uniqueness.

For the further properties of weak solution to the parabolic equations with the non-divergence free drift  we refer to  \cite{Krylov_4},  \cite{NU}, \cite{Bian} and references there.
The elliptic case is studied much better, see, for example \cite{Kim_Kim}, \cite{Tsai},  \cite{Krylov_2}, \cite{Krylov_3},  \cite{Kwon_1}, \cite{Kwon_2}, \cite{Lee} and references there. Note that such properties as the H\" older continuity and  higher integrability  of weak solutions require  generally speaking  more regularity   than only $b\in L_2(Q_T)$.  Namely, local boundedness of weak solutions for parabolic problems  is discussed in \cite{Albritton}, \cite{Filonov_Hodunov}, \cite{NU}.  H{\" o}lder continuity of weak  solutions to the problem \eqref{Equation}  with a drift $b$ belonging  to  a certain scale-invariant space was  studied, for example,  in    \cite{SSSZ}, \cite{Vicol}, \cite{SVZ} in  the divergence free case \eqref{Div-free_drift},   in \cite{ChSh-1}, \cite{NU}  in the non-spectral case \eqref{Non-spectral} and in \cite{ChSh-2}, \cite{ChSh-3},  \cite{NU} for certain  drifts in the spectral case $\div b\ge 0$. On the other hand, for drifts satisfying \eqref{Div-free_drift} it was shown  in  \cite{Filonov} in the elliptic case and in \cite{Bian} in the parabolic case  that the  continuity of bounded weak solutions to the problem \eqref{Equation} can fail if the drift is  supercritical.

Our paper is organized as follows. In Section \ref{Auxiliary_results} we introduce some auxiliary results which are basic tools of our paper. In Section
\ref{Existence} we   prove  Theorems \ref{Existence_Theorem} and \ref{Maximum_Princip_Theorem}.
In Section \ref{Stability} we present the proofs of Theorems \ref{Main_Theorem} and \ref{Stability_Theorem}. Finally, in Section \ref{Section_Counterexamples} we comment on Propositions \ref{Non-uniqueness} and \ref{Non-stability}.

\medskip

In the paper we use the following notation. For any $a$, $b\in
\mathbb  R^n$ we denote by $a\cdot b = a_k b_k$ their  scalar product in
$\mathbb R^n$. Repeated indexes assume the summation from 1 to $n$.   For $f$ depending on $x$ and $t$ we denote by $\nabla f$ the gradient of $f$ with respect to $x$ and $\cd_t f:=\frac{\cd f}{\cd t}$. We denote by $L_p(\Omega)$ and $W^k_p(\Omega)$ the
usual Lebesgue and Sobolev spaces. We do not distinguish between functional spaces of scalar and vector functions and omit the target space in notation. $C_0^\infty(\Om)$ is the space of
smooth functions compactly supported in $\Om$. The space
$\overset{\circ}{W}{^1_p}(\Omega)$ is the closure of
$C_0^\infty(\Omega)$ in $W^1_p(\Omega)$ norm. We denote by $\mathcal D'(\Om)$   the spaces of distributions on $\Om \subset \Bbb R^n$. For $p>1$ we denote $p'=\frac p{p-1}$.  The space $W^{-1}_p(\Om)$ is dual to $\overset{\circ}{W}{^1_{p'}}(\Om)$. The space $\operatorname{Lip}(\Om)$  consists of functions which are Lipschitz continuous of $\Om$.
 The symbols $\rightharpoonup$, $\overset{*}{\rightharpoonup}$ and $\to $ stand for the
weak, $*$-weak and strong convergence respectively. We denote by $B_R(x_0)$
the ball in $\mathbb R^n$ of radius $R$ centered at $x_0$ and write
$B_R$ if $x_0=0$. We write   $B$ instead of $B_1$.  For an open set $\Om_0\subset \Bbb R^n$ we write $\Om\Subset \Om_0$ if $\bar \Om$ is compact and $\bar \Om \subset \Om_0$.
For $p\in [1, +\infty)$ we denote by
$L_{p,w}(\Om)$ the weak Lebesgue space equipped with the norm
$$
\| b\|_{L_{p,w}(\Om)} \ := \ \sup\limits_{s>0}
s~|\{~x \in \Om: ~|b(x )|>s~\}|^{\frac 1p}.
$$
For $Q_T:=\Om\times (0,T)$ and $p\in [1, +\infty]$ we use the following notation for anisotropic Sobolev spaces:
$$
\gathered
W^{1,0}_p(Q_T) : = \{  u\in L_p(Q_T): \, \nabla u \in L_p(Q_T) \},  \  \| u\|_{W^{1,0}_p(Q_T)} = \| u\|_{L_p(Q_T)} +\|\nabla u\|_{L_p(Q_T)},   \\ W^{1,1}_p(Q_T) : = \{  u\in W^{1,0}_p(Q_T): \,  \cd_t u   \in L_p(Q_T) \}, \  \| u\|_{W^{1,1}_p(Q_T)} = \| u\|_{W^{1,0}_p(Q_T)} +\|\cd_t u\|_{L_p(Q_T)}, \\
W^{2,1}_p(Q_T) : = \{  u\in W^{1,1}_p(Q_T): \, \nabla^2 u   \in L_p(Q_T) \}, \  \| u\|_{W^{2,1}_p(Q_T)} = \| u\|_{W^{1,1}_p(Q_T)} +\|\nabla^2 u\|_{L_p(Q_T)}  .
\endgathered
$$
For a Banach space $X$  we denote by $X^*$ the dual space of $X$ with the duality relation $\langle u, w\rangle$, $u\in X$, $w\in X^*$.  The spaces $L_p(0,T; X)$, $p\in [1, +\infty)$  and
$L_\infty (0, T; X)$ are  the space of Banach-valued functions $u:(0,T)\to X$  equipped with the norms
$$
\| u\|_{L_p(0,T; X)} \ := \ \Big(\, \int\limits_0^T \| u(t)\|_X^p \, dt \, \Big)^{1/p}, \qquad
\| u\|_{L_\infty(0,T; X)} \ := \ \esssup\limits_{t\in (0,T)} \| u(t)\|_X.
$$
 $C([0,T] ; X) $ is the space of continuous Banach-valued functions $u:[0,T]\to X$, i.e.
$$
\forall\, t_0\in [0,T] \qquad \| u(t)-u(t_0)\|_X\to 0 \quad \mbox{as} \quad t\to t_0.
$$
The space
$C_w([0, T]; X)$ consists of  Banach-valued functions $u:[0,T]\to X$  which are weakly continuous in $t\in [0,T]$, i.e.
$$
\forall\, w\in X^* \quad  \mbox{the map} \quad t\in [0,T] \ \mapsto  \ \langle u(t), w\rangle \quad \mbox{is continuous on } [0,T].
$$
  For $u: \Om\to \Bbb R$ and $k\in \Bbb R$ we denote  $(u-k)_+:=\max\{ u-k, 0\}$, $(u-k)_-:=  \max\{ k-u, 0\}$.

\medskip
\noindent
{\bf Acknowledgement}. The research  of Timofey Shilkin was partly done during his stay  at the Max Planck Institute for Mathematics in the Sciences (MiS) in Leipzig in 2023-2024. The author thanks  MiS for the hospitality.

\newpage

\section{Auxiliary results}\label{Auxiliary_results}
\setcounter{equation}{0}

\bigskip
\bigskip
The following proposition is proved in \cite[Theorem 1.1]{Porretta}:
\begin{proposition}\label{Strong_continuity_in_time}
Assume $b$, $u_0$, $f$ satisfy \eqref{Problem_Data} and let $u$ be a weak solution to \eqref{Equation} in the sense of Definition \ref{Def_weak_solution}. Then $$\cd_t u\in L_2(0,T; W^{-1}_{2}(\Om)) + L_1(Q_T)$$
and hence
$$
u\in C([0,T]; L_1(\Om)).
$$
\end{proposition}

The next proposition (see, for example,  \cite[Chapter 3,  Lemma 1.4]{Temam}) says that weak solutions in the sense of Definition \ref{Def_weak_solution} are continuous in $t$ with respect to a weak topology in $L_2(\Om)$:
\begin{proposition}\label{Weak_continuity_in_time}
Assume $b$, $u_0$, $f$ satisfy \eqref{Problem_Data}  and let $u$ be a weak solution to \eqref{Equation} in the sense of Definition \ref{Def_weak_solution}. Then $$
u\in C_w([0,T] ; L_2(\Om))$$
which means that we can redefine $u\in L_\infty(0,T; L_2(\Om))$  so that $u(\cdot, t)\in L_2(\Om)$ for any $t\in [0,T]$ and for any $w\in L_2(\Om)$ the function
$$
 t\in [0,T] \ \ \mapsto \ \ \int\limits_\Om u(x,t)w(x)\, dx \quad \mbox{is continuous on } \ [0,T].
$$
\end{proposition}

The  next result  shows that we can reformulate definition of weak solutions so that the integral identity is satisfied for almost all moments of time:
\begin{proposition}\label{Integral_identity_pointwise}
Assume $b$, $u_0$, $f$ satisfy \eqref{Problem_Data} and let $u$ be a weak solution to \eqref{Equation} in the sense of Definition \ref{Def_weak_solution}. Then    for any $w\in \overset{\circ}{W}{^1_2}(\Om)\cap L_\infty(\Om)$  the
function
$$
t\in [0,T] \ \mapsto \ (u(t),w):=\int\limits_\Om u(x,t)w(x)\, dx \quad \mbox{is absolutely continuous}.
$$
Moreover, for a.e. $t\in (0,T)$ the
following identity holds:
$$
\gathered
\frac{d}{dt}\int\limits_\Om u(x,t) w(x)\, dx    +   \int\limits_\Om \nabla u(x,t)\cdot  \Big( \nu\,\nabla w(x) + b(x,t)\,  w(x)  \Big)\, dx \ =  \, \int\limits_{\Om} f(x,t)w(x)\, dx,\\
\forall w \ \in \overset{\circ}{W}{^1_2}(\Om)\cap L_\infty(\Om).
\endgathered
$$

\end{proposition}

Now we define the Steklov averaging in time variable.
Assume $T' \in (0,T)$.  For any $u\in L_2(Q_T)$  and any $0<h<T-T'$ we define
\begin{equation}
u_h (x,t)  \, := \, \frac 1h \, \int\limits_t^{t+h} u(x,\tau)\, d\tau, \qquad x\in \Om, \qquad t\in (0,T').
\label{Steklov_average}
\end{equation}
The following properties of the Steklov averaging are well-known, see \cite[Chapter II, Lemma 4.7]{LSU}:

\begin{proposition}
 Assume $u \in W^{1,0}_2(Q_T)$. Then for any $0<h<T-T'$ we have $$u_h \in W^{1,1}_2(Q_{T'})\cap C([0,T']; W^1_2(\Om))$$ and for a.e. $(x,t)\in Q_{T'}$
$$
\nabla u_h (x,t) \ = \ (\nabla u)_h (x,t),  \qquad \cd_tu_h (x,t) =    \frac{u(x,t+h)-u(x,t)}h.
$$
Moreover, the following convergence holds:
$$
\|   u_h -   u \|_{W^{1,0}_2(Q_{T'})} \ \to \ 0 \qquad \mbox{as} \quad h\to +0.
$$
\end{proposition}

For the Steklov averages $u_h$ the following analogue of Proposition \ref{Integral_identity_pointwise} holds:
\begin{proposition}\label{Integral_identity_pointwise_h}
Assume $b$, $u_0$, $f$ satisfy \eqref{Problem_Data} and let $u$ be a weak solution to \eqref{Equation} in the sense of Definition \ref{Def_weak_solution}. Take arbitrary $T' \in (0, T )$ and $ h \in (0, T - T')$. Denote by $u_h$ the Steklov averaging of $u$ defined in \eqref{Steklov_average}.  Then for a.e.  $t\in (0,T')$ the
following identity holds: 
\begin{equation}
\gathered
 \int\limits_\Om \cd_t u_h(x,t) w(x)\, dx  + \int\limits_\Om \Big(\nu\,\nabla u_h(x,t)\cdot  \nabla w(x) + (b\cdot \nabla u)_h(x,t)\,  w(x)\Big)\, dx \, = \, \int\limits_{\Om} f_h(x,t)w(x)\, dx,  \\  \forall\, w\in \overset{\circ}{W}{^1_2}(\Om)\cap L_\infty(\Om).
\endgathered \label{II_h}
\end{equation}
\end{proposition}

Now we define the truncation operator. For any $\dl>0$   denote by ${\mathcal T}_\dl \in \operatorname{Lip}(\Bbb R)$ the function
\begin{equation}
 {\mathcal T}_\dl(s)  \ := \ \left\{ \begin{array}{cl} 0, & s\le 0  \\ s, & 0<s< \dl \\  \dl , & s\ge \dl. \end{array} \right.
\label{Truncation_operator}
\end{equation}
Then  the composition ${\mathcal T}_\dl\circ u$  with a Sobolev function $u$ possesses the following properties:

\begin{proposition}
 Assume $u \in W^{1,1}_2(Q_T)$, Then for any $\dl>0$  we have  $${\mathcal T}_\dl(u) \in W^{1,1}_2(Q_T)\cap L_\infty(Q_T)$$ and
 $$
\nabla {\mathcal T}_\dl(u) \ = \ {\mathcal T}_\dl'(u)\, \nabla u, \qquad \cd_t {\mathcal T}_\dl (u) = {\mathcal T}_\dl'(u)\, \cd_t u \qquad \mbox{a.e. in} \quad Q_T.
$$
Moreover, if   $u_h$ is  defined by \eqref{Steklov_average} then
for any $T'\in (0,T)$
$$
\| {\mathcal T}_\dl(u_h) - {\mathcal T}_\dl(u) \|_{L_2(Q_{T'})} \ \to \ 0 \qquad \mbox{as} \quad h\to +0.
$$
\end{proposition}

Finally, we formulate the condition on boundedness  of  an integrable function, see a similar result in \cite[Theorem 6.1]{LSU}:

 \begin{proposition}\label{Level_sets_Theorem}
  Assume $u$ belongs to the energy class \eqref{Energy_class} and assume that there exists $k_0\ge 0$, $F> 0$, $\ga>0$ such that for any $k\ge k_0$ the following estimate holds:
\begin{equation}\label{Level_sets}
  \esssup\limits_{0<t<T} \int\limits_{\Om} |(u-k)_+|^2\, dx \ + \ \int\limits_{Q_T}|\nabla (u-k)_+|^2\, dxdt \ \le \ F^2\,  |A_k|^{1-\frac 2{n+2}+2\ga},
\end{equation}
  where
  $$
  A_k \, := \, \{ \, (x,t)\in Q_T: \, u(x,t)> k\, \}.
  $$
   Then $(u-k_0)_+\in L_\infty(Q_T)$ and
\begin{equation}\label{Boundedness}
  \| (u-k_0)_+\|_{L_\infty(Q_T)} \ \le  \ c\,  F ,
\end{equation}
  with some constant $c>0$ depending only on $n$, $\Om$, $T$ and $\ga$.
\end{proposition}

\begin{proof}
 Assume \eqref{Level_sets} holds. Then for $k\ge k_0$ by the H{\" o}lder inequality we have
 $$
 \int\limits_{A_k} (u-k)_+\, dxdt \ \le \ |A_k|^{1-\frac n{2(n+2)}} \, \|(u-k)_+\|_{L_{2+\frac 4n}(Q_T)} .
 $$
 From the parabolic imbedding
\begin{equation}
L_{\infty}(0,T; L_2(\Om))\cap W^{1,0}_2(Q_T)\hookrightarrow L_{2+\frac 4n}(Q_T)
\label{Parabolic_imbedding}
\end{equation}
  we have
 $$
 \| (u-k)_+\|_{L_{2+\frac 4n}(Q_T)}  \ \le \ c\, \Big( \|(u-k)_+\|_{L_\infty(0,T;L_2(\Om))} + \|\nabla (u-k)_+\|_{L_2(Q_T)}\Big)
 $$
 with some constant $c>0$ depending only on $n$, $\Om$ and $T$. Combining this inequality with \eqref{Level_sets} we arrive at
 $$
 \int\limits_{A_k} (u-k)_+\, dxdt \ \le \ c\, F\, |A_k|^{1+\ga}, \qquad \forall \, k\ge k_0,
 $$
 which gives \eqref{Boundedness}.
\end{proof}

\newpage

\section{Existence and boundedness of weak solutions}\label{Existence}
\setcounter{equation}{0}

\bigskip
The natural approach to prove the existence of a weak solution to the problem \eqref{Equation} is to approximate the drift $b\in L_2(Q_T)$ satisfying the non-spectral condition \eqref{Non-spectral} by smooth functions $b^m$.  Some technical difficulty is related to the fact that   one needs to preserve the condition \eqref{Non-spectral}. In the case of the divergence free drifts this difficulty can be   overcomed if we introduce the vector potential of $b$, extend this vector potential outside $\Om$ and mollify  the extended function. But in the case of  the condition \eqref{Non-spectral} besides the divergence free part of the drift  one needs also to extend the potential part of $b$. This  leads to a non-trivial problem of the extension of a subharmonic function. To  avoid it  we approximate    in Proposition   \ref{Approximation_Theorem}  below  not only the drift, right-hand side  and the initial data of the problem \eqref{Equation}, but also the domain $\Om$, see similar  ideas in \cite{Kwon_3}, \cite{Lee}, \cite{Verchota_1}, \cite{Verchota_2}.
The advantage of this approach is that now we can  mollify the drift  $b\in L_2(Q_T)$   so that the non-spectral condition \eqref{Non-spectral} is preserved on   subdomains of $Q_T$. Note that this construction allows to control  not only Lebesgue norms of the approximations $b^m$, but also some more complicated norms such as critical Morrey-type norms, see, for example, \cite{ChSh-1}.

Assume  Lipschitz domains
$\Om_m \subset\Bbb R^n$   satisfy  the following properties:
\begin{equation}
\gathered
\Om_{m}\subset \Om_{m+1}\subset \Om, \qquad \forall\, m\in \Bbb N, \qquad   \bigcup\limits_{m=1}^\infty \Om_m = \Om.
\endgathered
\label{Omega_m}
\end{equation}
(Here we also  allow the case $\Om_m=\Om$ for all $m\in \Bbb N$). Then the following proposition holds:

\begin{proposition}
  \label{Approximation_Theorem}
Assume $\{ \Om_m\}_{m=1}^\infty$ satisfy \eqref{Omega_m}, $u_0^m\in L_2(\Om)$, $f^m\in L_2(0,T; W^{-1}_2(\Om))$.
 Let $b^m\in C^\infty(\bar Q_T)$ satisfy the assumption
\begin{equation}
\div b^m \ \le \ 0 \quad \mbox{in} \quad  Q_T^m:= \Om_m\times (0,T).
\label{Non-spectral_m}
\end{equation}
Denote by $u^m$ the unique weak solution to the problem
\begin{equation}\label{Equation_m}
\left\{ \quad \gathered  \cd_t u^m  -\nu\,\Delta u^m + b^m\cdot \nabla u^m   \ = \  f^m \qquad \mbox{in} \quad Q^m_T, \\
u^m|_{\cd\Om_m\times (0,T)} \ = \ 0, \qquad \qquad  \qquad \\
u^m|_{t=0} \ = \ u^m_0, \qquad \qquad  \qquad  \endgathered \right.
\end{equation}
and extend $u^m$ by zero from $Q_T^m $ onto $\Bbb R^n\times (0,T)$. Then the following estimate holds
\begin{equation}\label{Energy_estimate_m}
\|  u^m\|_{L_\infty(0,T; L_2(\Om))}  \ + \  \sqrt{\nu}\, \| \nabla u^m\|_{L_2(Q_T)}  \, \le \, c\, \Big( \|u_0^m\|_{L_2(\Om)} + \|f^m \|_{L_2(0,T; W^{-1}_2(\Om))} \Big),
\end{equation}
where the constant $c>0$ depends only on $n$, $\Om$ and $T$.
\end{proposition}

\begin{proof}
The proof of Proposition \ref{Approximation_Theorem} is standard and can be obtained by testing the equations \eqref{Equation_m} by $u^m$. As the drift $b^m\in C^\infty(\bar Q_T)$ satisfies \eqref{Non-spectral_m} integrating by parts we obtain
$$
\int\limits_{\Om_m} b^m\cdot \nabla u^m \, u^m \, dx  \ = \ - \frac 12\, \int\limits_{\Om_m} \div b^m |u^m|^2\, dx \ \ge \ 0
$$
and hence we can drop this term. So, we arrive at \eqref{Energy_estimate_m}. We omit the details.
\end{proof}

Now we can prove Theorem \ref{Existence_Theorem}.

\begin{proof} Consider any sequence of Lipschitz domains $\Om_{m}\Subset \Om_{m+1} \Subset \Om$ satisfying \eqref{Omega_m}. Denote by $\om_\ep(x,t)$ a standard Sobolev kernel with respect to $(x,t)$-variable:
$$
\om_\ep(x,t) = \tfrac 1{\ep^{n+1}}\, \om\left(\tfrac x\ep, \tfrac t\ep\right), \qquad \om \in C^\infty_0(B \times \Bbb R), \qquad \int\limits_{-\infty}^{+\infty}\int\limits_{\Bbb R^n} \om(x,t)\, dx dt =1.
$$
Take $\ep_m\to +0$ so that $\ep_m <\dist \{ \bar \Om_m, \cd \Om \}$ and denote by $b^m$ the convolution
$$
b^m(x,t) = \int\limits_{-\infty}^{+\infty}\int\limits_{\Bbb R^n} \om_{\ep_m}(x-y,t-\tau)\,  b(y,\tau)\, dy d\tau.
$$
Then $b^m\in C^\infty(\bar Q_T)$ satisfies \eqref{Non-spectral_m} and
\begin{equation}\label{Conv_1}
\| b^m - b\|_{L_2(Q_T)} \to 0 \qquad \mbox{as} \qquad m\to \infty.
\end{equation}
Assume also  $u^m_0 \in C^\infty(\bar \Om)$ and $f^m\in C^\infty(\bar Q_T)$ satisfy
\begin{equation}\label{Conv_2}
\| u^m_0 - u_0\|_{L_2(\Om)} \to 0, \qquad
\| f^m - f\|_{L_2(0,T; W^{-1}_2(\Om))} + \| f^m -f\|_{L_1(Q_T) }\to 0
\end{equation}
as $m\to \infty$.
Let $u^m$ be weak solutions to the problem \eqref{Equation_m}. Using \eqref{Energy_estimate_m} we  conclude there exists $u$  belonging to \eqref{Energy_class}  and  a subsequence $u^m$ such that
\begin{equation}\label{Conv_3}
u^m  \ \rightharpoonup \  u \qquad \mbox{in} \quad W^{1,0}_2(Q_T).
\end{equation}
Take any $\eta \in C_0^\infty(\Om\times [0,T))$. Then for sufficiently large $m$ we obtain $\operatorname{supp} \eta \subset \Om_m \times [0, T)$. Using \eqref{Conv_1}, \eqref{Conv_2}, \eqref{Conv_3} we can pass to the limits in the identity
$$
\int\limits_{Q_T} \Big( -u^m\cd_t\eta + \nu \, \nabla u^m\cdot \nabla \eta \Big)\, dxdt \ + \ \mathcal B[u^m,\eta] \ = \ \int\limits_\Om u_0^m(x)\eta(x,0)\, dx \ + \ \int\limits_{Q_T}  f^m \eta \, dxdt
$$
  and arrive at \eqref{Integral_Identity}. Hence $u$ is a weak solution to the problem \eqref{Equation}.  The energy estimate follows automatically from \eqref{Energy_estimate_m}.
\end{proof}

\newpage

\section{Proof of Theorems  \ref{Main_Theorem} and \ref{Stability_Theorem}}\label{Stability}
\setcounter{equation}{0}

\bigskip

We start from the proof of  Theorem \ref{Main_Theorem}.

 \begin{proof}
Take arbitrary  $T'\in (0, T)$ and $\dl>0$. For any $h\in (0,T-T')$ denote by $u_h$ the Steklov average of $u$ defined in \eqref{Steklov_average} and denote by ${\mathcal T}_\dl$ the truncation operator defined in \eqref{Truncation_operator}.
Then for any $t\in [0,T']$ we have
$$
{\mathcal T}_\dl (u_h(\cdot ,t))  \in \overset{\circ}{W}{^1_2}(\Om)\cap L_\infty(\Om).
$$
Hence we can fix a.e. $t\in (0,T')$ and take $w= {\mathcal T}_\dl( u_h(\cdot, t)) $ in the identity  \eqref{II_h}. For a.e. $t\in (0,T')$ we obtain
\begin{equation}
\gathered
\int\limits_{\Om}\Big(  {\mathcal T}_\dl(u_h)\cd_t  u_h  +\nu\, \nabla u_h\cdot  \nabla {\mathcal T}_\dl (u_h ) + (b \cdot \nabla u)_h  {\mathcal T}_\dl( u_h )\Big)\, dx  =  \int\limits_{\Om} f_h {\mathcal T}_\dl( u_h)\, dx.
\endgathered
\label{II_with_T}
\end{equation}
We claim that the first term in \eqref{II_with_T} for a.e. $t\in (0, T')$  satisfies the identity
\begin{equation}
\int\limits_{\Om}   {\mathcal T}_\dl(u_h(x,t))\cd_t  u_h(x,t) \, dx  \ = \ \frac{d}{dt}\Big( \tfrac 12\, \| {\mathcal T}_\dl(u_h(\cdot, t))\|_{L_2(\Om)}^2 + \dl \, \|(u_h(\cdot, t)-\dl)_+\|_{L_1(\Om)}\Big).
\label{Main_Lemma}
\end{equation}
Indeed, taking into account the identity
$$
\cd_t {\mathcal T}_\dl( u_h) = {\mathcal T}_\dl'(u_h)\, \cd_t u_h \qquad \mbox{a.e. in} \quad Q_{T'}
$$
for a.e. $t\in (0,T')$ we obtain the relation
\begin{equation}
\gathered
\frac 12\,  \frac{d}{dt}\| {\mathcal T}_\dl(u_h(\cdot, t))\|_{L_2(\Om)}^2 \, = \, \int\limits_{D_\dl^h(t)} {\mathcal T}_\dl(u_h(x,t)) \cd_t u_h(x,t)\, dx \, = \\ =  \, \int\limits_{\Om} {\mathcal T}_\dl(u_h(x,t)) \cd_t u_h(x,t)\, dx \, - \, \dl \, \int\limits_{A_\dl^h(t)} \cd_t u_h(x,t) \, dx,
 \endgathered
 \label{1}
 \end{equation}
 where we denote
 $$
 D_\dl^h(t):= \{\, x\in \Om: ~0 \le u_h(x,t)  \le \dl \, \}, \qquad A_\dl^h (t):= \{ \, x\in \Om:~u_h(x,t) > \dl \, \}.
 $$
 On the other hand,  taking into account $(u_h-\dl)_+\in W^{1,1}_2(Q_{T'})$ and
 $$
 \cd_t (u_h-\dl )_+ \, = \, \chi(u_h -\dl )\, \cd_t u_h \qquad \mbox{a.e. in}\quad Q_{T'},
 $$
 where we denote by $\chi(s)$ the standard Heaviside step function,
 for a.e. $t\in (0,T')$  we obtain
\begin{equation}
  \frac{d}{dt}\int\limits_\Om (u_h(x,t)-\dl)_+\, dx \, = \, \int\limits_{A_\dl^h(t)} \cd_t u_h(x,t)\, dx.
  \label{2}
 \end{equation}
  Combining \eqref{1} and \eqref{2} we obtain \eqref{Main_Lemma}.

  Note that the second  term in \eqref{II_with_T} is non-negative:
  \begin{equation}
  \nu\,\int\limits_{\Om}   \nabla u_h\cdot  \nabla {\mathcal T}_\dl (u_h )\, dx \, = \,\nu\, \int\limits_{\Om} |\nabla {\mathcal T}_\dl(u_h)|^2\, dx \ \ge \  0.
  \label{Viscos_term}
  \end{equation}
  So, taking into account \eqref{Main_Lemma} and \eqref{Viscos_term} from \eqref{II_with_T} for a.e. $t\in (0, T' )$ we obtain the inequality
  $$
  \frac{d}{dt}\Big( \tfrac 12\, \| {\mathcal T}_\dl(u_h(\cdot, t))\|_{L_2(\Om)}^2 + \dl \, \|(u_h(\cdot, t)-\dl)_+\|_{L_1(\Om)}\Big) \ \le \ - \int\limits_\Om (b \cdot \nabla u)_h  {\mathcal T}_\dl( u_h ) \, dx \ + \  \int\limits_{\Om} f_h {\mathcal T}_\dl( u_h)\, dx.
  $$
  Taking arbitrary $0\le t_1< t_2\le T'$ and integrating the last inequality over $t\in (t_1,t_2)$ we arrive at
\begin{equation}
  \gathered
  \tfrac 12\, \| {\mathcal T}_\dl(u_h(\cdot, t_2))\|_{L_2(\Om)}^2 + \dl \, \|(u_h(\cdot, t_2)-\dl)_+\|_{L_1(\Om)} \ \le \\ \le  \ \tfrac 12\, \| {\mathcal T}_\dl(u_h(\cdot, t_1))\|_{L_2(\Om)}^2 + \dl \, \|(u_h(\cdot, t_1)-\dl)_+\|_{L_1(\Om)} \ -    \   \int\limits_{Q_{t_1, t_2}} (b \cdot \nabla u)_h  {\mathcal T}_\dl( u_h ) \, dxdt \ + \\ +  \ \int\limits_{Q_{t_1, t_2}} f_h {\mathcal T}_\dl( u_h)\, dxdt,
  \endgathered
  \label{3}
  \end{equation}
   where we denote $Q_{t_1, t_2}:= \Om \times (t_1, t_2)$.
  Dropping the first non-negative term in the left-hand side of \eqref{3} and taking into account the inequalities
  $$
  {\mathcal T}_\dl(s)\le \dl \qquad \Longrightarrow \qquad  \| {\mathcal T}_\dl(u_h(\cdot, t_1))\|_{L_2(\Om)}^2 \, \le \, \dl^2 \, |\Om|
  $$
  from \eqref{3} we obtain
 \begin{equation}
  \gathered
  \|(u_h(\cdot, t_2)-\dl)_+\|_{L_1(\Om)} \ \le  \ \tfrac 12\,  \dl  \, |\Om| \, +  \, \|(u_h(\cdot, t_1)-\dl)_+\|_{L_1(\Om)} \ - \\ -  \  \frac 1\dl\,  \int\limits_{Q_{t_1, t_2}} (b \cdot \nabla u)_h  {\mathcal T}_\dl( u_h ) \, dxdt \ +  \  \frac 1\dl\,  \int\limits_{Q_{t_1, t_2}} f_h {\mathcal T}_\dl( u_h)\, dxdt. 
  \endgathered
  \label{4}
  \end{equation}
As $u\in C([0,T]; L_1(\Om))$ for any $t\in  [0, T']$ we have the strong convergence
$$
\| u_h(\cdot ,t) - u(\cdot,t )\|_{L_1(\Om)} \ \to \ 0 \qquad \mbox{as} \quad h\to +0
$$
and hence
\begin{equation}
\forall\, t\in [0,T'] \qquad  \|(u_h(\cdot, t)-\dl)_+\|_{L_1(\Om)} \ \to \  \|(u(\cdot, t )-\dl)_+\|_{L_1(\Om)} \quad  \qquad \mbox{as} \quad h\to +0 . 
  \label{5}
\end{equation}
On the other hand, we have
\begin{equation}
\int\limits_{Q_{t_1, t_2}} (b \cdot \nabla u)_h \,  {\mathcal T}_\dl( u_h ) \, dx dt  \ \to \ \int\limits_{Q_{t_1, t_2}}   b \cdot \nabla u  \,  {\mathcal T}_\dl( u  ) \, dx  dt \qquad \mbox{as} \quad h\to +0,
\label{Convective_term}
\end{equation}
\begin{equation}\label{RHS_h_to_zero}
   \int\limits_{Q_{t_1, t_2}} f_h {\mathcal T}_\dl( u_h)\, dxdt \ \to \ \int\limits_{Q_{t_1, t_2}}  f\,  {\mathcal T}_\dl( u)\, dxdt \qquad \mbox{as} \quad h\to +0.
\end{equation}
Indeed, we have
$$
\gathered
 \int\limits_{Q_{t_1, t_2}} (b \cdot \nabla u)_h \,  {\mathcal T}_\dl( u_h ) \, dxdt \ -  \int\limits_{Q_{t_1, t_2}}  b \cdot \nabla u \,  {\mathcal T}_\dl( u) \, dxdt  \ = \\ = \
 \int\limits_{Q_{t_1, t_2}} \Big( (b \cdot \nabla u)_h - b \cdot \nabla u\Big) \,  {\mathcal T}_\dl( u_h  ) \, dxdt \ + \
 \int\limits_{Q_{t_1, t_2}} b \cdot \nabla u  \,  \big( {\mathcal T}_\dl( u_h ) - {\mathcal T}_\dl(u)\big) \, dx dt.
\endgathered
$$
As $b\cdot \nabla u\in L_1(Q_T)$ for any $T'\in (0,T)$ we have $$(b \cdot \nabla u)_h\to b \cdot \nabla u \quad \mbox{in} \quad L_1(Q_{T'}) \quad  \mbox{as} \quad h\to +0. $$
So, for $0\le t_1< t_2\le T'$ we obtain
$$
\gathered
\Big|\, \int\limits_{Q_{t_1, t_2}} \Big( (b \cdot \nabla u)_h - b \cdot \nabla u\Big) \,  {\mathcal T}_\dl( u_h  ) \, dxdt \, \Big| \ \le \ \dl \, \|(b \cdot \nabla u)_h - b \cdot \nabla u\|_{L_1(Q_{T'})} \ \underset{h\to +0}{\to} \ 0.
\endgathered
$$
For the second term we have ${\mathcal T}_\dl(u_h)\to {\mathcal T}_\dl(u)$ in $L_1(Q_{T'})$ and hence for a subsequence we obtain ${\mathcal T}_\dl(u_h)\to {\mathcal T}_\dl(u)$ a.e. in $Q_{T'}$. As
$$
 \Big|\, b \cdot \nabla u  \,  \big( {\mathcal T}_\dl( u_h ) - {\mathcal T}_\dl(u)\big)\, \Big| \ \le \ 2\dl ~|b \cdot \nabla u | \quad \mbox{a.e. in} \quad Q_{T'}, \qquad b \cdot \nabla u \in L_1(Q_{T'})
$$
by Lebesgue's dominated convergence theorem we obtain
$$
\int\limits_{Q_{t_1, t_2}}  b \cdot \nabla u  \,  \big( {\mathcal T}_\dl( u_h ) - {\mathcal T}_\dl(u)\big) \, dx dt \ \to \ 0 \qquad \mbox{as} \quad h\to +0
$$
and hence \eqref{Convective_term} follows. Passing to the limit in \eqref{4} and taking into account \eqref{5}, \eqref{Convective_term}   and \eqref{RHS_h_to_zero}  we arrive at the inequality
\begin{equation}
  \gathered
   \|(u (\cdot, t_2)-\dl)_+\|_{L_1(\Om)} \ \le   \ \tfrac 12\,  \dl  \, |\Om| \, +  \, \|(u (\cdot, t_1)-\dl)_+\|_{L_1(\Om)} \ - \\ -  \  \frac 1\dl\, \int\limits_{Q_{t_1, t_2}}   b \cdot \nabla u~ {\mathcal T}_\dl( u  ) \, dxdt \ + \  \frac 1\dl\, \int\limits_{Q_{t_1, t_2}}   f\,  {\mathcal T}_\dl( u  ) \, dxdt.
  \endgathered
\label{6}
\end{equation}
We have
$$
\int\limits_{Q_{t_1, t_2}}   b \cdot \nabla u~ {\mathcal T}_\dl( u  ) \, dxdt \ = \ \int\limits_{Q_{t_1, t_2}\cap D_\dl}  b \cdot \nabla u~ {\mathcal T}_\dl( u  ) \, dxdt
\ + \ \dl \, \int\limits_{Q_{t_1, t_2}\cap A_\dl}  b \cdot \nabla u  \, dxdt,
$$
where we denote
$$
D_\dl \, := \, \{\, (x,t)\in Q_T: ~0 \le  u(x,t) \le \dl\, \}, \qquad A_\dl \, := \, \{\, (x,t)\in Q_T: ~u(x,t) >  \dl\, \}. 
$$
Taking into account the identities
$$
\tfrac 12\, \nabla |{\mathcal T}_\dl(u)|^2 \, = \, {\mathcal T}_\dl(u)\, \nabla u   \quad \mbox{a.e. in} \quad D_\dl, \qquad
\nabla (u-\dl)_+ \, = \, \nabla u \quad \mbox{a.e. in} \quad A_\dl
$$
we obtain the relation
\begin{equation}
  \int\limits_{Q_{t_1, t_2}}  b \cdot \nabla u~ {\mathcal T}_\dl( u  ) \, dxdt \ = \ \tfrac 12\,   \int\limits_{Q_{t_1, t_2}} b \cdot \nabla |{\mathcal T}_\dl( u  )|^2 \, dxdt
\ + \ \dl \,   \int\limits_{Q_{t_1, t_2}}  b \cdot \nabla (u-\dl)_+  \, dxdt. 
\label{7}
\end{equation}
Applying  the condition \eqref{Non-spectral} we finally obtain
$$
\int\limits_{Q_{t_1, t_2}} b \cdot \nabla u~ {\mathcal T}_\dl( u  ) \, dxdt \ \ge  \ 0. 
$$
Hence we can drop this term in \eqref{6}.
From ${\mathcal T}_\dl(s)\le \dl$ we also conclude
$$
 \frac 1\dl\,  \int\limits_{Q_{t_1, t_2}}  f\,  {\mathcal T}_\dl( u  ) \, dxdt \ \le \  \int\limits_{t_1}^{t_2} \int\limits_{\Om[u(t) > 0 ]}  |f|\, dx dt,
$$
where we denote $\Om[u(t) > 0 ]:=\{ \, x\in \Om:\, u(x,t)>0\,\}$. So,  for any $\dl>0$ we obtain
$$
\|(u (\cdot, t_2)-\dl)_+\|_{L_1(\Om)} \ \le   \ \tfrac 12\,  \dl  \, |\Om| \, +  \, \|(u (\cdot, t_1)-\dl)_+\|_{L_1(\Om)} \ + \  \int\limits_{t_1}^{t_2} \int\limits_{\Om[u(t) > 0]} |f|\, dx dt.
$$
Taking the limit as $\dl \to +0$  and varying $T'$  for any $0\le t_1<t_2< T$ we arrive at
\begin{equation}\label{Decay_L_1_norm_plus}
  \|(u (\cdot, t_2))_+\|_{L_1(\Om)} \ \le   \    \|(u (\cdot, t_1) )_+\|_{L_1(\Om)} \, + \,  \| f\|_{L_1(Q_{t_1, t_2} [ u > 0])},
\end{equation}
where we denote   $Q_{t_1, t_2}[u>0]:= \{ \, (x,t)\in Q_{t_1, t_2}: \, u(x,t)>0\, \}$.
Applying the same technique to the function $-u$ we obtain also
\begin{equation}\label{Decay_L_1_norm_minus}
    \|(u (\cdot, t_2))_-\|_{L_1(\Om)} \ \le   \    \|(u (\cdot, t_1) )_-\|_{L_1(\Om)} \, + \,  \| f\|_{L_1(Q_{t_1, t_2} [u < 0])}. 
\end{equation}
Hence \eqref{Decay_L_1_norm} follows.
\end{proof}

Now we can prove Theorem \ref{Stability_Theorem}.

\begin{proof}
Denote $v^m:= u^m-u$.
Then $v^m$ is a weak solution to the problem
\begin{equation}\label{Equation_v_m}
\left\{ \quad \gathered  \cd_t v^m  -\nu\, \Delta v^m + b^{ m}\cdot \nabla  v^m   \ = \  g^m  \qquad \mbox{in} \quad Q_T, \\
v^m|_{\cd\Om\times (0,T)} \ = \ 0, \qquad \qquad  \qquad \\
v^m|_{t=0} \ = \ v^m_0, \qquad \qquad  \qquad  \endgathered \right.
\end{equation}
where we denote   $$g^m:= f^m-f - (b^m - b)\cdot \nabla u  \ \in \  L_1(Q_T),$$
$$ v^m_0:= u^m_0-u_0 \ \in \ L_2(\Om).$$
 Taking arbitrary $t\in (0,T)$ and applying Theorem \ref{Main_Theorem} with $t_1=0$, $t_2=t$  we obtain the estimate
 $$
 \| v^m (\cdot, t)\|_{L_1(\Om)} \ \le \ \| v^m_0  \|_{L_1(\Om)} \ + \ \, \| g^m\|_{L_1(Q_T)}. 
 $$
For the last term we have the estimate
$$
\gathered
  \| g^m\|_{L_1(Q_T)} \ \le \ \| f^m - f\|_{L_1(Q_T)} + \| b^m - b\|_{L_2(Q_T)} \|\nabla u \|_{L_2(Q_T)}
 \endgathered
$$
which along with the energy estimate gives \eqref{Stability_estimate}.
 \end{proof}

 \newpage
 \section{Proof of Theorem  \ref{Maximum_Princip_Theorem}}
 \setcounter{equation}{0}

\medskip
In this section we  prove  Theorem \ref{Maximum_Princip_Theorem}. Note that for  $f \equiv 0$  Theorem \ref{Maximum_Princip_Theorem} follows already from Theorem \ref{Main_Theorem}  as  in this case we have for $k_0:= \| u_0\|_{L_\infty(\Om)}$ and $t_1=0$
$$(u(\cdot, 0) -k_0)_{+}  = 0, \qquad  (u(\cdot, 0) + k_0)_{-}  = 0, $$
from
\eqref{Decay_L_1_norm_plus} and \eqref{Decay_L_1_norm_minus} we obtain
$$(u(\cdot, t) -k_0)_{+}  = 0, \qquad  (u(\cdot, t) + k_0)_{-}  = 0 \qquad \mbox{for any} \quad t>0. $$

\begin{proof}  Denote by  $b^m\in C^\infty(\bar Q_T)$,  $u_0^m \in C^{\infty}(\bar \Om)$  and $f^m\in C^\infty(\bar Q_T)$ the smooth approximations of $b$, $u_0$ and $f$ constructed in Section \ref{Existence} in the proof of Theorem \ref{Existence_Theorem}. Note that we can construct approximations in such a way that the following inequalities hold:
\begin{equation}\label{Data_approximation}
\| u_0^m \|_{L_\infty(\Om)} \, \le \, \| u_0  \|_{L_\infty(\Om)}, \qquad  \|f^m\|_{L_s(Q_T)} \, \le  \, \|f\|_{L_s(Q_T)}.
\end{equation}
Denote by $u^m$ the solutions to the problem  \eqref{Equation_m}. Denote $k_0= \| u_0\|_{L_\infty(\Om)}$. For  $k>k_0$ we can take $\eta:=(u^m-k)_+$ as a test function in the integral identity \eqref{Integral_Identity} for the solution $u^m$ to the problem \eqref{Equation_m}.
For the convective term using \eqref{Non-spectral_m} we obtain  the inequality
$$
\int\limits_{\Om_m} b^m\cdot \nabla u^m \, \eta \, dx  =  \int\limits_{\Om_m} b^m\cdot \nabla \tfrac 12 \, |(u^m-k)_+|^2\, dx \ + \ k\, \int\limits_{\Om_m} b^m\cdot \nabla (u^m-k)_+\, dx \ \ge \ 0
$$
and hence this term can be dropped. Integrating in time and taking into account   $\eta (x,0)=0$   we obtain
$$
 \tfrac 12 \esssup\limits_{t\in (0,T)} \int\limits_{\Om} |(u^m-k)_+|^2\, dx \ + \ \nu \int\limits_{Q_T}|\nabla (u^m-k)_+|^2\, dxdt   \ \le \ \int\limits_{Q_T}   |f^m|\, (u^{m}-k)_+\, dxdt. 
$$
Applying the H{\" o}lder inequality we get
$$
 \int\limits_{Q_T}   |f^m|\, (u^{m}-k)_+\, dxdt \ \le \ \| f^m\|_{L_{q}(A_k)} \| (u^m-k)_+\|_{L_{2+ \frac 4n}(Q_T)},
$$
where
$$
q:= \frac{2(n+2)}{n+4}, \qquad
  A_k \, := \, \, \{\, (x,t)\in Q_T: \, u^m(x,t)>k\, \} .
$$
Using the imbedding \eqref{Parabolic_imbedding} for the energy class
 and the Young inequality  we arrive at
$$
 \esssup\limits_{t\in (0,T)} \int\limits_{\Om} |(u^m-k)_+|^2\, dx \ + \ \int\limits_{Q_T}|\nabla (u^m-k)_+|^2\, dxdt  \ \le \ c\,  \| f^m\|_{L_q(A_k)}^2
$$
with some constant $c>0$ depending on $n$, $\Om$, $T$ and $\nu$. As $s>\frac {n+2}2>q$ by the H{\" o}lder inequality we obtain
$$
\| f^m\|_{L_q(A_k)}^2 \ \le \ \| f^m\|_{L_{\frac {n+2}2}(A_k)}^2 |A_k|^{1 - \frac 2{n+2}} \ \le  \ \| f^m \|_{L_{s}(Q_T)}^2 \, |A_k|^{1 - \frac 2{n+2}+2\ga},
$$
where $\ga = \frac 2{n+2}- \frac 1s$. Taking into account \eqref{Data_approximation} for any $k\ge k_0$ we arrive at
$$
  \esssup\limits_{0<t<T} \int\limits_{\Om} |(u^m-k)_+|^2\, dx \ + \ \int\limits_{Q_T}|\nabla (u^m-k)_+|^2\, dxdt \ \le \ c\, \| f \|_{L_{s}(Q_T)}^2 \,  |A_k|^{1-\frac 2{n+2}+2\ga}
$$
with some constant $c>0$ depending only on $n$, $\Om$, $T$, $\nu$ and $s$.
By Proposition \ref{Level_sets_Theorem} we obtain
$$
\| (u^m-k_0)_-\|_{L_\infty(Q_T)} \ \le \ c\, \|f\|_{L_s(Q_T)}  .
$$
Similarly we can get
$$
\| (u^m-k_0)_-\|_{L_\infty(Q_T)} \ \le \ c\, \|f\|_{L_s(Q_T)}.
$$
These estimates together give
$$
\| u^m \|_{L_\infty(Q_T)} \ \le \ \| u_0\|_{L_\infty(\Om)} \ + \  c\, \|f\|_{L_s(Q_T)}.
$$
Hence there exists $v\in L_\infty(Q_T)$ such that for some subsequence we have
$$
u^m \overset{*}{\rightharpoonup} v \qquad \mbox{in} \quad L_\infty(Q_T), \qquad \| v\|_{L_\infty(Q_T)} \, \le \, \| u_0\|_{L_\infty(\Om)} \ + \  c\, \|f\|_{L_s(Q_T)}.
$$
It is clear also that $v$ must be a weak solution to the problem \eqref{Equation}. From uniqueness of weak solutions we conclude $v\equiv u$ a.e. in $Q_T$.
\end{proof}

\newpage
\section{Counterexamples}\label{Section_Counterexamples}
\setcounter{equation}{0}

\bigskip
In this section we discuss  the proof of  Proposition \ref{Non-uniqueness}. In the case of the whole space $\Om=\Bbb R^n$ the non-trivial weak solution $u\not\equiv 0$  (in the sense of Definition \ref{Def_weak_solution}) to the Cauchy problem
\begin{equation}\label{Equation_Cauchy_problem}
\left\{ \quad \gathered  \cd_t u  -   \Delta u + b\cdot \nabla u   \ = \ 0  \qquad \mbox{in} \qquad \Bbb R^n \times (0,T), \\
u|_{t=0} \ = \ 0, \qquad \qquad  \qquad  \endgathered \right.
\end{equation}
can be obtained easily if we multiply the fundamental solution of the heat equation $$\Ga (x,t)=(4\pi t)^{-\frac n2} e^{-\frac{|x|^2}{4t}}, \qquad x\in \Bbb R^n, \quad t>0, $$ by some positive  power of $t$. Indeed, assume
$$
u(x,t) \ = \ t^{-\al} \,  e^{-\frac{|x|^2}{4t}}
$$
and
$$
b(x,t) \ = \ \left(n-2\al\right)\, \frac{x}{|x|^2}.
$$
It is clear that
$$
b\in L_\infty(0,T; L_{n,w}(\Bbb R^n))
$$
and  $u$ and $b$ satisfy the parabolic equation \eqref{Equation_Cauchy_problem} everywhere in $ \Bbb R^n \times (0,T)$. Moreover,
\begin{equation}\label{Growing_norms}
\| u(\cdot, t)\|_{L_2(\Bbb R^n)}^2 \ = \ c_n \, t^{\frac n2 -2\al}, \qquad \| \nabla u(\cdot ,t)\|_{L_2(\Bbb R^n)}^2 \ = \   c_n \, t^{\frac n2-1 -2\al}
\end{equation}
and hence for any
\begin{equation}\label{Restriction_on_alpha}
\al  \, < \, \tfrac n4
\end{equation}
we obtain
$$
u \in L_\infty(0,T; L_2(\Bbb R^n)), \qquad \nabla u \in L_2(\Bbb R^n \times (0,T))
$$
and
$$
\| u(\cdot , t ) \|_{L_2(\Bbb R^{ n })} \ \to \ 0 \qquad \mbox{as} \qquad t\to +0.
$$
So, for any $\al$ satisfying \eqref{Restriction_on_alpha} $u$ is a weak solution to the problem \eqref{Equation_Cauchy_problem}. In particular, one can take $\al=0$. To construct a non-trivial weak solution to the problem \eqref{Equation} in the case of a bounded domain $\Om$  we can localize the construction above. So, we present the proof of Proposition \ref{Non-uniqueness}.

\begin{proof} Assume $\Om=\{\, x\in \Bbb R^n:\, |x|<1\, \}$. Take
$$
  u(x,t) \ = \ \zeta(x,t)~e^{-\frac{|x|^2}{4t}},
  \qquad \zeta(x,t)  \ = \ \zeta_0(|x|,t), \qquad \zeta_0(r,t) \ = \ e^{\frac {r-1}t } -1
$$
and
 $$
  b(x,t)  \ = \ b_0(|x|, t)\, \frac {x}{|x|^2}, \qquad
b_0(r,t)  \ = \ n - 1   \  -  \ \frac{\zeta_0}{\frac 2r \left(\zeta_0+ 1\right)- \zeta_0 }.
$$
Note that
$$
-1\le \zeta_0(r, t) \le 0, \qquad \forall\, r\in [0,1], \qquad \forall\, t>0,
$$
and hence
$$
-1 \ \le \ \frac{\zeta_0}{\frac 2r \left(\zeta_0+ 1\right)- \zeta_0 } \ \le \ 0,   \qquad \forall \, r\in [0,1], \quad \forall\, t>0.
$$
This implies that
$$
b\in L_\infty(0,T; L_{n,w}(\Om)).
$$
On the other hand, the direct computations show that
 $$
 \cd_ t u -\Delta u +b\cdot \nabla  u  \ = \  \left[  \left( \cd_t \zeta_0  -    \cd_r^2\zeta_0  + \frac{|x|}{t}\, \cd_r\zeta_0\right)  +   (b_0-n + 1)  \left(    \frac{\cd_r\zeta_0}{|x|}     - \frac{\zeta_0}{2t} \right)  \  +  \  \frac{\zeta_0} {2t} \right]~e^{-\frac{|x|^2}{4t}}.
$$
As $\zeta_0$ and $b_0$ satisfy
$$
\cd_t \zeta_0  - \cd_r^2\zeta_0 + \frac rt\, \cd_r\zeta_0  \ = \ 0, \qquad b_0 - n + 1 \ = \  - \frac{ \frac{\zeta_0} {2t}}{   \frac{\cd_r\zeta_0}{r}     - \frac{\zeta_0}{2t}}
$$
we conclude that
$u$ satisfies \eqref{Equation} everywhere in $Q_T$.  Finally, it is obvious that $u|_{\cd\Om\times (0,T)} =0$ and from $|\zeta(x,t)|\le 1$ we get
$$
\| u(\cdot, t)\|_{L_2(\Om)}^2 \ \le \ \int\limits_{\Bbb R^n } e^{-\frac{|x|^2}{2t}} \, dx \ = \ c_n t^{\frac n2} \ \to \ 0 \qquad \mbox{as} \qquad t\to+0.
$$
Moreover,
$$
\nabla u(x,t) =  \frac 1t \,  e^{-\frac{(|x|-2)^2}{4t}} \frac{x}{|x|} -    \frac{x}{2t} \Big(e^{\frac{|x| - 1}{t}} -1\Big) e^{-\frac{|x|^2}{4t}}.
$$
As $|x|\le 1$ we obtain $(|x|-2)^2\ge 1 $ and hence
$$
    \|\nabla u (\cdot, t) \|^2_{L_2(\Omega)} \le c_n e^{-\frac{1}{ 2 t}} t^{-2} + c_n t^{\tfrac n2 - 1}, \qquad \forall\, t>0.
$$
So, we derive
$$
    \nabla u  \in L_2(Q_T).
$$
 Hence $u$ is a weak solution to the problem \eqref{Equation} corresponding to $u_0\equiv 0$ and $f\equiv 0$.

\end{proof}

Now we comment on the proof of  Proposition \ref{Non-stability}.

\begin{proof}
Take  any $a \in (1,2)$ and consider $ \dl  \to +0$. Then for    $t_{ \dl } = { \dl }^{ -a }$  we obtain
$$
   \|b^{ \dl } -b \|_{L_2( Q_{t_{ \dl }} )} \, \le \,  c_n   { \dl }^{ 1 - a/2 } \to 0, \qquad  \|u^{ \dl }(\cdot, t_{ \dl }) - u_0  \|_{L_1(\Om)} \, \ge \,  c_n(e^{{ \dl }^{1 - a}} - 1)  \to +\infty
$$
with some positive  constant $c_n$ depending only on $n$. This relations give  \eqref{Norm_growth}
for sufficiently small $ \dl >0$ (depending on $\ep$).
\end{proof}

\end{document}